\documentclass[11pt]{amsart}

\usepackage[all,color]{xy}
\usepackage{pb-diagram}
\usepackage[mathscr]{eucal}
\usepackage{hyperref}
\hypersetup{
    colorlinks=true, 
    linktoc=all,     
    linkcolor=blue,  
}

\usepackage{xcolor}

\usepackage[toc,page]{appendix}
\usepackage{pifont}
\usepackage{combelow} 

\RequirePackage{ifthen,setspace,enumitem,tikz}
\usetikzlibrary{arrows}

\DeclareMathAlphabet{\mathpzc}{OT1}{pzc}{m}{it}

\usepackage{amsfonts}
\usepackage{amsmath}
\usepackage{amssymb}
\usepackage{needspace}

\newcommand{\tr}{\textnormal{tr}}

\newcommand{\dbar}{\overline{\partial}}

\newcommand{\Ric}[2]{R_{#1 \overline{#2}}}

\newcommand{\diam}{\textnormal{Diam}}

\newcommand{\ddbar}{\sqrt{-1}\partial\dbar}

\def\Vol{\textnormal{Vol}}

\newtheorem{theorem}{Theorem}[section]
\newtheorem{proposition}{Proposition}[section]
\newtheorem{lemma}{Lemma}[section]

\newtheorem{corollary}{Corollary}[section]

\newtheorem{conjecture}{Conjecture}[section]

\numberwithin{equation}{section}

\renewcommand{\Ric}{\operatorname{Ric}}

\setlist[enumerate]{leftmargin=*,label=\textup{(\roman*)},itemsep=3pt}
\allowdisplaybreaks[1]
\hypersetup{pdftitle={Type I estimates for the Kahler-Ricci flow II},
 pdfauthor={Wangjian Jian and Jian Song},
 pdfsubject={Diameter, Gromov-Hausdorff limits, and Type I tangent flows; working draft}}
\title[Type I estimates for the K\"ahler-Ricci flow II]{Type I estimates for the K\"ahler-Ricci flow II}
\author{Wangjian Jian}
\address{Institute of Mathematics, Academy of Mathematics and Systems Science,
Chinese Academy of Sciences, Beijing, 100190, China}
\email{wangjian@amss.ac.cn}
\thanks{Wangjian Jian is supported in part by NSFC grants 12422103, 12201610,
12371058, and 12288201; BJNSF grant JR25002; and National Key R\&D Program
of China grants 2023YFA1009900 and 2021YFA1003100.}
\author{Jian Song}
\address{Department of Mathematics, Rutgers University,
Piscataway, NJ 08854, USA}
\email{jiansong@math.rutgers.edu}
\thanks{Jian Song is supported in part by the National Science Foundation
grant DMS-2505575.}

\date{}

\newcommand{\GH}{\mathrm{GH}}

\begin{document}
\raggedbottom
\begin{abstract}
In this paper, we establish geometric consequences of our Type I scalar
curvature estimate for finite-time solutions of the K\"ahler-Ricci flow
on compact K\"ahler manifolds. We prove a uniform global diameter bound
and a Type I diameter bound for every fibre of finite time holomorphic contraction. These estimates
yield Gromov-Hausdorff convergence to a unique compact metric space
homeomorphic to the normal K\"ahler space determined by the limiting
cohomology class. We also establish a local partial $C^0$ estimate for
the normalized flow and show that the tangent shrinkers are normal
complex analytic varieties with log terminal singularities.
\end{abstract}
\maketitle

\section{Introduction}\label{sec:introduction}

This paper is a sequel to \emph{Type I estimates for the
K\"ahler-Ricci flow I} \cite{JSTypeI}, where we established a Type I
bound for the scalar curvature along finite-time solutions of the
K\"ahler-Ricci flow on compact K\"ahler manifolds. Our purpose here is
to derive the geometric consequences of this bound.

The K\"ahler-Ricci flow provides a natural analytic approach to the
classification of K\"ahler manifolds. In the analytic minimal model
program proposed by Song and Tian
\cite{STKodaira,STCanonical,STSingularities}, the flow is expected to
carry out the birational surgeries of the minimal model
program, its finite-time singularities corresponding to contractions
determined by the evolving cohomology class. A basic problem is
therefore to understand the geometry of the flow as it approaches a
singular time: whether the metrics converge, what the limit is, and
what the singularities look like at their natural parabolic scale. We
shall see that the Type I bound for the scalar curvature is strong
enough to provide substantial answers to these questions. It yields
global and fibre diameter bounds and Gromov-Hausdorff convergence at
the singular time, and, at the Type I scale, a local partial
$C^0$-estimate together with the complex analytic structure of blow-up
limits.

\medskip
We now describe the setting more precisely.
Let $X$ be a compact   K\"ahler manifold of complex dimension
$n\ge2$, and let $\omega_0$ be a smooth K\"ahler metric on $X$. The
K\"ahler-Ricci flow evolves the metric by its Ricci form and preserves
the K\"ahler condition. We consider its maximal unnormalized solution
\begin{equation}\label{eq:flow}
 \partial_t\omega=-\Ric(\omega),\qquad \omega(0)=\omega_0,
 \qquad 0\le t<T<\infty.
\end{equation}
We write $g(t)=2\omega(t)(\,\cdot\,,J\,\cdot\,)$ and 
$$ d_t=d_{g(t)}, $$
following the normalization of Part I. The maximal existence time is
characterized cohomologically by
\begin{equation}\label{eq:maximaltime}
 T=\sup\bigl\{t>0:\ [\omega_0]-2\pi t c_1(X)
                 \text{ is a K\"ahler class}\bigr\}.
\end{equation}
This characterization is due to Tian and Zhang \cite{TZMax},
extending earlier work of Tsuji \cite{Tsuji}. We stress that no
rationality assumption on $[\omega_0]$ and no projectivity assumption
on $X$ will be imposed.

The first step is to identify the expected limit of the flow. The
limiting cohomology class
\[
 \alpha=\lim_{t\rightarrow T}[\omega(t)]
       =[\omega_0]-2\pi T c_1(X)
\]
is nef but not K\"ahler, and it is natural to expect that it governs
the geometry of the flow at time $T$. At the level of complex geometry,
this is made precise by the transcendental base-point-free theorem of
Hacon and Xie \cite[Theorem~1.4]{HX}. In the smooth case needed here,
it asserts that a nef real $(1,1)$-class $\lambda$ with
$\lambda-c_1(K_X)$ K\"ahler is the pullback of a K\"ahler class under
a holomorphic contraction. Since
\[
 \frac{\alpha}{2\pi T}-c_1(K_X)=\frac{[\omega_0]}{2\pi T},
\]
we obtain a surjective holomorphic map with connected fibres
\begin{equation}\label{eq:contraction}
 \Phi:X\longrightarrow Y,\qquad \alpha=\Phi^*[\theta_Y],
\end{equation}
where $Y$ is a compact normal K\"ahler space and $\theta_Y$ is a
K\"ahler form on $Y$.   Here the identity holds in real Bott-Chern
cohomology, and smooth forms on $Y$ are understood through local
ambient potentials, so that $\theta=\Phi^*\theta_Y$ is smooth and
semipositive on $X$. In particular, if $Y$ is a point, we set $\theta=0$. Throughout
the paper, fibre diameters are measured with respect to the ambient
distance $d_t$ on $X$.

The contraction \eqref{eq:contraction} indicates how the flow should
degenerate: the curvature should blow up at most at the Type I rate,
while the fibres of $\Phi$ should collapse at the corresponding
parabolic rate. This was formulated by Jian, Song and Tian
\cite[Conjecture~1.1]{JST} for rational initial classes; we state
their conjecture in the present K\"ahler setting.

\begin{conjecture} \label{conj:JST}
For the flow \eqref{eq:flow} and the contraction \eqref{eq:contraction},
there exists a constant $C>0$ such that
\[
 |\mathrm{R}_{g(t)}(x)|\le\frac{C}{T-t},\qquad
 \diam_{d_t}\Phi^{-1}(y)\le C(T-t)^{1/2}
\]
for all $x\in X$, $y\in Y$ and $0\le t<T$.
\end{conjecture}

The model case is that of a Fano manifold with
$[\omega_0]=2\pi T c_1(X)$, for which $Y$ is a point. After the
rescaling
\[
 \widehat\omega(s)=(T-t)^{-1}\omega(t),\qquad
 s=-\log\frac{T-t}{T},
\]
the flow becomes the normalized K\"ahler-Ricci flow
$\partial_s\widehat\omega=-\Ric(\widehat\omega)+\widehat\omega$
in $2\pi c_1(X)$. In this setting, Perelman's entropy and
noncollapsing theory \cite{Perelman} yields uniform bounds for the
scalar curvature, the diameter, and the normalized Ricci potential and
its gradient. A detailed account of Perelman's work on the Fano Ricci flow was given by Sesum and Tian
\cite{SesumTian}. In particular, writing
$\widehat g(s)=(T-t)^{-1}g(t)$, Perelman's diameter estimate becomes
\[
 \diam(X,g(t))=(T-t)^{1/2}\diam(X,\widehat g(s))
 \le C(T-t)^{1/2},
\]
which is precisely the fibre estimate of Conjecture~\ref{conj:JST}
when $Y$ is a point.

Beyond diameter bounds, one expects the flow itself to converge, and
the limit to be the contraction $Y$, now viewed as a metric space.
This prediction for the unrescaled metrics is part of Song and Tian's
analytic minimal model program \cite[Section~6.2]{STSingularities}. Special cases of smooth divisorial contractions 
were established by Song and Weinkove
\cite{SW,SWII,SWNotes}, with further metric flips and higher-codimensional
birational surgeries  achieved by Song and Yuan \cite{SYFlips}
and Song \cite{SongSurgery}. In the volume-noncollapsing case, Collins
and Tosatti \cite[Theorem~1.5]{CollinsTosatti} identified the
singularity locus with the null locus of the limiting class. Related results on finite-time collapse
and subsequential Gromov-Hausdorff convergence were obtained for
projective bundles by Song, Sz\'ekelyhidi and Weinkove \cite{SSW}
and for certain Fano bundles by Fu and Zhang \cite{FuZhang}.

\begin{conjecture} \label{conj:STmetric}
As $t\rightarrow T$, the unnormalized K\"ahler-Ricci flow
\eqref{eq:flow} converges in the Gromov-Hausdorff topology to a compact
metric space homeomorphic to $Y$, the normal K\"ahler space induced
by the limiting cohomology class through \eqref{eq:contraction}.
\end{conjecture}

At the Type I scale, the Fano case again provides the model, and here
the theory is considerably more developed. A central tool is the
partial $C^0$-estimate, introduced by Tian \cite{Tian90}, which
relates metric convergence to algebraic geometry through holomorphic
sections. A further result in this direction was obtained by
Donaldson and Sun \cite{DS}. Along the normalized Fano K\"ahler-Ricci flow, Tian and Zhang
\cite{TZFano} established partial $C^0$-estimates in arbitrary
complex dimension assuming the Hamilton--Tian conjecture,
and proved the conjecture in complex dimension at most three. In arbitrary
dimension, the Hamilton-Tian conjecture was established by Bamler \cite{BamBounded}
through his compactness and regularity theory for Ricci flows
with bounded scalar curvature, and by Chen and Wang
\cite{CWII}.

The question of uniqueness has also been settled in the Fano case.
Following earlier work of Tian-Zhang \cite{TZFano}, Dervan-Sz\'ekelyhidi
\cite{DervanSzekelyhidi} and Chen-Sun-Wang \cite{CSWFano}, which relates the limit of the flow to an
optimal degeneration of $X$, Han and Li \cite{HanLi} established the
algebraic uniqueness of the optimal degeneration, showing in
particular that the limit is independent of the initial metric in
$2\pi c_1(X)$. These results suggest that, in the general finite-time
setting, each pointed contraction should likewise admit a canonical
shrinking soliton model. Motivated by the two-step degeneration
picture of Sun and Zhang \cite[Conjectures~6.4 and 6.8]{SunZhang}, we
formulate the following conjecture.

\begin{conjecture} \label{conj:optimalshrink}
For every fixed $p\in X$, there exists a unique complete, possibly
singular, K\"ahler-Ricci shrinker $(Z,g_Z,p_\infty)$ such that
\[
 \bigl(X,(T-t)^{-1}g(t),p\bigr)
 \xrightarrow[t\rightarrow T]{\text{pointed Gromov-Hausdorff}}
 (Z,g_Z,p_\infty).
\]
The shrinker is the algebraic optimal degeneration of the pointed
contraction $(X\xrightarrow{\Phi}Y,p)$.
\end{conjecture}

Hallgren and Zhang \cite{HZ}
recently established the quasi-projectivity of singular
K\"ahler-Ricci shrinkers arising as noncollapsed limits of compact
K\"ahler-Ricci flows. More precisely, each tangent shrinker $Z$ in
Theorem~\ref{thm:tangent} admits a polarized Fano fibration $Z\to W$, where
$Z$ is a normal quasi-projective variety with log terminal
singularities, $W$ is a normal affine variety, and the morphism is
projective with connected fibres and relatively ample $-K_Z$.

Conjecture~\ref{conj:optimalshrink} generalizes the compact soliton
and optimal-degeneration picture for the normalized Fano
K\"ahler-Ricci flow \cite{CSWFano,DervanSzekelyhidi,HanLi}. Here the relevant
contraction germ is taken over $\Phi(p)\in Y$, while $p$ is retained
as the marked point. In the projective setting, 'optimal
degeneration' refers to the K-polystable polarized Fano fibration
obtained from a weighted-volume-minimizing valuation and its two-step
degeneration, as proposed by Sun and Zhang
\cite[Conjectures~6.4 and 6.8]{SunZhang}. For general K\"ahler initial
data, the conjecture asks for the corresponding construction from the
local analytic contraction. Theorem~\ref{thm:tangent} below
establishes the existence and normal complex analytic structure of
subsequential tangent shrinkers, and thus provides an analytic
foundation for this conjecture. In view of this progress, particularly with the Type I scalar curvature estimates, we believe that
Conjecture~\ref{conj:optimalshrink} is now well within reach, and
expect it to be resolved in the near future.

\subsection{Diameter estimates}\label{subsec:introdiameter}

We now describe our main results. The first is a uniform bound for the
diameter of $X$, together with a Type I bound for the diameters of all
contracted fibres.

\begin{theorem} \label{thm:diameter}
Let $d_t=d_{g(t)}$ be the Riemannian distance on $X$ induced by $g(t)$.
Then there exists a constant $C>0$, depending only on the initial data
and the fixed background choices, such that for all $0\le t<T$:
\begin{enumerate}
\item[(i)] The diameter of $X$ is uniformly bounded
\begin{equation}\label{gldiam}
 \diam(X,g(t))\le C.
\end{equation}
\item[(ii)] Every contracted fibre satisfies the Type I diameter bound
\begin{equation}\label{eq:fibreintro}
 \diam_{d_t}\Phi^{-1}(y)\le C(T-t)^{1/2},
 \qquad for ~each~ y\in Y.
\end{equation}
\end{enumerate}
\end{theorem}

Together with the scalar curvature estimate of Part I,
Theorem~\ref{thm:diameter}(ii) establishes Conjecture~\ref{conj:JST}
in full in the present K\"ahler setting.   When $Y$ is a point, the bound (ii) reduces to Perelman's
diameter estimate \cite{SesumTian} for Fano manifolds. In general, it
extends that estimate from the whole manifold to every contracted
fibre of $\Phi$. 

In the volume-noncollapsing case, the uniform diameter bound in
Theorem~\ref{thm:diameter}(i) was previously established by
Guo, Phong, Song and Sturm \cite{GPSSDiameter}. The two parts of Theorem~\ref{thm:diameter} are proved by rather
different methods. The global bound (i) follows from a uniform
H\"older estimate for the distance functions $d_t$ with respect to a
fixed background metric. This estimate is obtained by comparing
$\omega(t)$ with auxiliary K\"ahler metrics in the fixed class
$[\omega_0]$, to which the geometric H\"older estimates of Guo, Song
and Sturm \cite{GSS} apply uniformly, even though the classes of the
metrics along the flow degenerate. The fibre bound (ii) combines two
ingredients: a comparison of distances at different times, derived
from heat kernel estimates and the Wasserstein contraction of backward
heat flows, and the vanishing in the limit of the energy of
holomorphic maps whose images lie in a fibre.

\subsection{Gromov-Hausdorff convergence}\label{subsec:introGH}

Our second main result confirms Conjecture~\ref{conj:STmetric} for
arbitrary initial K\"ahler metrics, in both the volume-collapsing and
the volume-noncollapsing cases.

\begin{theorem} \label{thm:GH}
As $t\rightarrow T$, $(X,g(t))$ converges in the Gromov-Hausdorff topology
to a compact metric space $(Y,d_T)$, where $d_T$ induces the
complex-analytic topology of $Y$.
\end{theorem}

In the volume-noncollapsing case, Gromov--Hausdorff convergence
to a unique limit was established by Li-Zhang
\cite[Theorem~1.1]{LiZhangGH}. Li-Zhang \cite[Section~4]{LiZhangShrinkers} also obtained an
analogous identification for the canonical flows of noncompact
K\"ahler-Ricci shrinkers. Theorem~\ref{thm:GH} covers
volume collapsing case and identifies the limit with the contraction
space $Y$ endowed with its complex-analytic topology. 

Thus the contraction selected by the limiting cohomology class is
realized by the metric limit of the flow as well. More precisely, as
$t\rightarrow T$, the distances $d_t$ converge uniformly on
$X\times X$ to the function $(p,q)\mapsto d_T(\Phi(p),\Phi(q))$, so
that the Gromov-Hausdorff convergence in Theorem~\ref{thm:GH} is
realized by the contraction $\Phi$ itself. The proof rests on the
diameter estimates as the fibre diameter bound identifies the points
which merge at time $T$, the Schwarz estimate separates distinct
fibres, and the uniform H\"older control of the distances ensures
that the resulting metric induces the original topology of $Y$.

In the noncollapsed case, the limit can be described more concretely.
If the limiting volume is positive, equivalently $\dim_{\mathbb C}Y=n$,
then $\Phi$ is bimeromorphic. Let $Y^\circ$ be the complement of
$Y_{\mathrm{sing}}$ and the critical values of $\Phi$ over
$Y_{\mathrm{reg}}$, and set $X^\circ=\Phi^{-1}(Y^\circ)$.
Under the biholomorphism $\Phi:X^\circ\to Y^\circ$, the flow converges
smoothly on compact subsets to a K\"ahler metric $g_T$ on $Y^\circ$,
whose intrinsic distance satisfies
\[
 d_T(y,y')\le d_T^\circ(y,y')\le C d_T(y,y')^\kappa,
 \qquad y,y'\in Y^\circ,
\]
for some $C>0$ and $\kappa\in(0,1)$. Thus the identity map from
$(Y^\circ,d_T^\circ)$ to $(Y^\circ,d_T)$ is $1$-Lipschitz and
bi-H\"older. These assertions are proved in Section~\ref{sec:GH}.

\subsection{Type I blow-up}\label{subsec:introblowup}

We finally turn to the geometry at the Type I scale, which describes the
local models of singularity formation. Fix $x\in X$, let
$t_i\rightarrow T$, and set $\tau_i=T-t_i$. We consider the rescaling
about the singular time
\begin{equation}\label{eq:blowupintro}
 g_i(s)=\tau_i^{-1}g(T+\tau_i s),\qquad -T/\tau_i\le s<0,
\end{equation}
under which the original time $t_i$ corresponds to $s=-1$ and the
base point remains $x$.  

Our third main result shows that these rescalings subconverge to
singular shrinking solitons which are normal complex analytic
varieties.

\begin{theorem} \label{thm:tangent}
For any fixed $x\in X$ and any sequence $\tau_i\rightarrow0$, a
subsequence of \eqref{eq:blowupintro} converges to a complete singular
shrinking K\"ahler-Ricci soliton flow with the following properties.
\begin{enumerate}
\item[(i)]   Let $(Z,d)$ denote its time $-1$
slice. Then there are points $z(s)\in Z$ such that, for every fixed $s<0$,
\begin{equation}\label{eq:intropointedGH}
 (X,d_{g_i(s)},x)\xrightarrow[i\to\infty]{\text{pointed Gromov-Hausdorff}}
 (Z,\sqrt{|s|}\,d,z(s)).
\end{equation}
The space $Z$ is independent of $s$ and the convergence is smooth on compact
subsets of the regular limiting spacetime.

\item[(ii)]  The dense
regular locus $\mathcal R_Z$ of $Z$ carries a metric $g_Z$ and a potential
$f$ satisfying
\begin{equation}\label{eq:introsoliton}
 \Ric_{g_Z}+\nabla^2_{g_Z} f=\frac12 g_Z,
\end{equation}
whose metric completion is $Z$. The singular set has real Minkowski
codimension at least four, and every metric tangent cone is a
Ricci-flat K\"ahler cone.

\item[(iii)]  The space $Z$ is a normal
complex analytic variety with log terminal singularities. Its
analytic topology agrees with its metric topology, its smooth locus
is $\mathcal R_Z$.

\end{enumerate}
\end{theorem}

If the original flow extends smoothly to time $T$ near $x$, the limit
is simply the Gaussian shrinking soliton.  We note that
Theorem~\ref{thm:tangent} establishes subsequential convergence and
the normal analytic structure of the limit, while the uniqueness and
optimal-degeneration assertions of Conjecture~\ref{conj:optimalshrink}
remain open.

It may be instructive to compare Theorem~\ref{thm:tangent} with the
Fano case. When $X$ is Fano and $[\omega_0]=2\pi T c_1(X)$, the Type I
rescalings correspond to the long-time behavior of the normalized Fano
flow, and the limiting shrinkers are compact. As recalled above,
Perelman's entropy, noncollapsing, scalar curvature and diameter
estimates \cite{Perelman,SesumTian} provide the basic analytic control
in that case. The Hamilton-Tian conjecture, asserting subsequential convergence
to shrinking K\"ahler-Ricci solitons with mild singularities, was
proved by Tian and Zhang \cite{TZFano} in complex dimensions at most
three. In arbitrary dimension, it was established by Bamler
\cite{BamBounded} through his theory of Ricci flows with bounded
scalar curvature, and by Chen and Wang \cite{CWII}. The partial $C^0$-estimates of Tian and Zhang \cite{TZFano},
valid in arbitrary dimension assuming Hamilton--Tian convergence,
provide the analytic input for the projective algebraic description
of the limits. Later building on the work of Chen-Sun-Wang
\cite{CSWFano} and Dervan-Sz\'ekelyhidi \cite{DervanSzekelyhidi}, Han and Li \cite{HanLi}
proved that the limit is optimal degeneration of the original manifold.
Theorem~\ref{thm:tangent} extends the existence and structure of these
compact models to the possibly noncompact tangent shrinkers of general
finite-time singularities.

The proof of Theorem~\ref{thm:tangent} draws on Bamler's entropy,
heat-kernel and compactness theory for Ricci flows
\cite{BamHeat,BamBounded,BamCompact,BamStructure}, on its K\"ahler
refinements by Hallgren and Jian \cite{HJ}, and on the complex
geometric compactness results of Jian, Song and Tian \cite{JST} and of
Hallgren, Jian, Song and Tian \cite{HJST}. The key point is that the
scalar curvature bound of Part I allows these results to be applied at
a fixed point of the original manifold. More specifically, the
shrinking $\mathbb F$ limit is supplied by Bamler's tangent-flow
theorem \cite[Theorem~2.37]{BamStructure}, applied to a conjugate heat
family based at the singular time and associated with $x$, and the
identification of its slices with $(Z,\sqrt{|s|}\,d)$ follows from
\cite[Theorem~2.18(a)-(c)]{BamStructure}. To upgrade this to pointed
Gromov-Hausdorff convergence, we use Hallgren's singular-time
heat-kernel estimates \cite[Lemma~3.2]{HallgrenEntropy} and his
metric-measure convergence criterion \cite[Proposition~2.7]{HallgrenGW},
together with the distance estimates of Bamler and Zhang \cite{BZII}. It is here that Part I supplies the required scalar curvature bounds.
Finally, the normal analytic structure, the log terminal property and
the equality of the analytic and metric regular loci follow from the
recent work of Hallgren and Zhang
\cite[Theorem~B(i),(iii); Section~2.2, case~(I)]{HZ}, and the soliton
identifications preserve this analytic structure.

Along the way, we also establish a local partial $C^0$-estimate for
the rescaled flows (Theorem~\ref{thm:partialC0}), adapting the
estimates for holomorphic sections and the peak-section construction
of \cite[Section~8]{JST}. A noteworthy feature is that the potential
constructed in Part I provides a local polarization of the rescaled
metrics near each fibre, so that no rationality assumption on the
initial class is needed.

\medskip
The paper is organized as follows. Section~\ref{sec:background}
collects the basic estimates. In Section~\ref{sec:globaldiameter}, we
prove the H\"older and global diameter bounds, and in
Section~\ref{sec:fibres}, the Type I fibre bound.
Section~\ref{sec:GH} is devoted to the Gromov-Hausdorff convergence in
Theorem~\ref{thm:GH}. Section~\ref{sec:partialC0} develops the local
partial $C^0$-estimate, and Section~\ref{sec:soliton} treats the
tangent soliton and its normal complex analytic structure.

\bigskip
\noindent\textbf{Acknowledgements.} We would like to thank Professor
Gang Tian for his generous support and encouragement over the years,
which have been a constant source of inspiration. This work was
carried out with the assistance of ChatGPT, which contributed to
discussions of the arguments and to the preparation and revision
of the manuscript. In particular, AI assistance led to a substantial
improvement of the authors' original proof of
Theorem~\ref{thm:diameter}(ii).

\bigskip

\section{Background}\label{sec:background}

After replacing $\omega(t)$ by $T^{-1}\omega(Tt)$, we may and shall
assume $T=1$. All constants below may depend on this normalization.
We retain the conventions of Part I:
\[
 g=2\omega(\,\cdot\,,J\,\cdot\,),\quad
 \partial_tg=-2\Ric_g,\quad
 \Delta=\Delta_g=\tr_\omega\ddbar,\quad
 R=R_g=\tr_\omega\Ric(\omega).
\]
Choose a smooth volume form $\Omega$ with
$\Ric(\Omega)=\omega_0-\theta$, and put
\[
 \omega_t=(1-t)\omega_0+t\theta,\qquad
 \omega(t)=\omega_t+\ddbar\varphi(t).
\]
The additive normalization of $\varphi$ is fixed by
\begin{equation}\label{eq:MA}
 \dot\varphi=\log\frac{(\omega_t+\ddbar\varphi)^n}{\Omega},
 \qquad\varphi(0)=0.
\end{equation}
Set $\lambda=1-t$ and $\square=\partial_t-\Delta$.

\begin{proposition}\label{prop:background}
There exists a constant $C>0$ such that, for all $0\le t<1$,
\begin{align}
 &\|\varphi(t)\|_\infty\le C,\qquad \dot\varphi\le C,
       \qquad 0\le\tr_{\omega(t)}\theta\le C,\label{eq:basic}\\
 &-C\le R(\cdot,t)\le\frac{C}{1-t}.\label{eq:scalarI}
\end{align}
For $U=(1-t)\dot\varphi+\varphi+nt$,
\begin{align}
 &\ddbar U=\omega-\theta-(1-t)\Ric(\omega),\label{eq:Ucurv}\\
 &U_t=n-(1-t)R,\quad \square U=\tr_\omega\theta,
 \quad |U|+|\partial U|_\omega^2\le C.\label{eq:Ubasic}
\end{align}
Moreover,
\begin{equation}\label{eq:Qidentity}
 \square(t\dot\varphi-\varphi)=n-\tr_\omega\omega_0,
 \qquad \ddot\varphi=-R.
\end{equation}
\end{proposition}

\begin{proof}
The $C^0$ estimate, the Schwarz estimate, and the scalar curvature
lower bound follow from \cite[Proposition~2.1]{JSTypeI}. The scalar
curvature upper bound is \cite[Theorem~1.1]{JSTypeI}, and the identities
and estimates for $U$ are given in \cite[Lemma~4.1]{JSTypeI}, with the
same normalization.
Differentiating \eqref{eq:MA} gives $\ddot\varphi=-R$, and the lower
bound for the scalar curvature then yields
$\dot\varphi(t)\le\dot\varphi(0)+Ct\le C$.
Finally,
\[
 \square\dot\varphi=\tr_\omega(\theta-\omega_0),\qquad
 \square\varphi=\dot\varphi-n+\tr_\omega\omega_t,
\]
which imply the remaining identity in \eqref{eq:Qidentity}.
\end{proof}

For a point $(p,t)$ with $t<1$ and an earlier time $0\le s<t$, write
\[
 d\nu_{p,t;s}(z)=K(p,t;z,s)\,dV_{g(s)}(z)
       =\bigl(4\pi(t-s)\bigr)^{-n}e^{-f_{p,t;s}(z)}\,dV_{g(s)}(z).
\]
Here $K$ is the conjugate heat kernel:
\[
 (-\partial_s-\Delta_{g(s)}+R(\cdot,s))K(p,t;\cdot,s)=0,
\]
\[
 \int_Xd\nu_{p,t;s}=1,\qquad
 \nu_{p,t;s}\rightharpoonup\delta_p\quad(s\rightarrow t).
\]
Perelman's entropy estimate gives the following uniform bounds for the
pointed Nash entropy (see \cite[Lemma~2.1]{JSTypeI} and \cite{Perelman}):
\[
 \mathcal N_{p,t}(t-s):=\int_X f_{p,t;s}\,d\nu_{p,t;s}-n,
 \qquad -A\le\mathcal N_{p,t}(t-s)\le0.
\]

We next recall the metric notions needed below. For probability measures
$\mu,\nu$ on a compact metric space $(M,d)$, their Wasserstein distance
of order one is
\begin{align*}
 d_{W_1}(\mu,\nu)
   &:=\inf_{\Lambda\in\Pi(\mu,\nu)}
       \int_{M\times M}d(z,w)\,d\Lambda(z,w)\\
   &=\sup_{\operatorname{Lip}_d u\le1}
      \left|\int_Mu\,d\mu-\int_Mu\,d\nu\right|,
\end{align*}
where $\Pi(\mu,\nu)$ is the set of probability measures with marginals
$\mu$ and $\nu$. Following \cite[Definition~3.10]{BamHeat}, set
\[
 H_{2n}=\frac{(2n-1)\pi^2}{2}+4.
\]
A point $(z_s,s)$ is an $H_{2n}$-center of $(p,t)$ if
\[
 \int_X d_s(z_s,z)^2\,d\nu_{p,t;s}(z)\le H_{2n}(t-s).
\]
Such a center exists on every earlier slice
\cite[Proposition~3.12]{BamHeat}. In fact, the variance estimate
\cite[Corollary~3.8]{BamHeat} bounds the average of the left-hand side
with respect to $d\nu_{p,t;s}(z_s)$ by $H_{2n}(t-s)$, so at least one
point attains this bound.

Bamler's Gaussian upper bound \cite[Theorem~7.2]{BamHeat}, together
with the entropy bound and $R\ge-C$, gives
\begin{equation}\label{eq:heatupper}
 K(p,t;z,s)\le C(t-s)^{-n}
       \exp\left(-\frac{d_s(z_s,z)^2}{C(t-s)}\right).
\end{equation}
The constants here are uniform, since $t-s\le1$ and the lower bound
for the scalar curvature therefore remains uniform after rescaling the
interval. Moreover, backward heat propagation contracts the
Wasserstein distance
\cite[Lemma~2.7]{BamHeat}:
\begin{equation}\label{eq:Wcontract}
 d_{W_1}^{g(s)}(\nu_{p,t;s},\nu_{q,t;s})\le d_t(p,q).
\end{equation}

Perelman's reduced distance from $(p,t)$ to $(z,s)$ is
\[
 \ell_{p,t}(z,s)=\frac{1}{2\sqrt{t-s}}\inf_{\zeta}
 \int_0^{t-s}\sqrt r\left(
    |\dot\zeta(r)|_{g(t-r)}^2+R(\zeta(r),t-r)\right)\,dr,
\]
where the infimum is over piecewise smooth curves with $\zeta(0)=p$
and $\zeta(t-s)=z$. The reduced-density comparison states
\[
 f_{p,t;s}(z)\le\ell_{p,t}(z,s),\qquad
 K(p,t;z,s)\ge\bigl(4\pi(t-s)\bigr)^{-n}e^{-\ell_{p,t}(z,s)};
\]
see \cite[Section~7 and Corollary~9.5]{Perelman}.

Finally, for $\tau_i\rightarrow0$ and a compact interval
$[a,b]\subset(-\infty,0)$, the rescaled flow
$g_i(s)=\tau_i^{-1}g(1+\tau_i s)$ satisfies
\begin{equation}\label{eq:rescaledscalar}
 -C\tau_i\le R_{g_i(s)}\le C/|s|.
\end{equation}
Indeed, scalar curvature rescales as
$R_{g_i(s)}=\tau_i R_{g(1+\tau_i s)}$, so this follows directly from
\cite[Theorem~1.1 and Proposition~2.1]{JSTypeI}.
Since $n\ge2$, Ye's finite-time Sobolev inequality
\cite[Theorem~D$^*$]{YeSobolev}, after this rescaling, reads
\[
 \left(\int |v|^{2n/(n-1)}dV_{g_i}\right)^{(n-1)/n}
 \le A\int\left(|dv|_{g_i}^2+\tfrac14R_{g_i}v^2\right)dV_{g_i}
       +B\tau_i\int v^2dV_{g_i}.
\]
Thus \eqref{eq:rescaledscalar} gives, uniformly for $s\in[a,b]$,
\begin{equation}\label{eq:sobolev}
 \left(\int |v|^{2\chi}dV_{g_i}\right)^{1/\chi}
 \le C\int\bigl(|dv|_{g_i}^2+|v|^2\bigr)dV_{g_i},
 \qquad \chi=\frac{n}{n-1}.
\end{equation}
This inequality will be applied locally, to functions supported in
balls of fixed radius. On these rescaled intervals, we shall also use
the volume noncollapsing and noninflating estimates at unit scale,
together with compactness under a scalar curvature bound
\cite{BZII,BamBounded,HallgrenEntropy}. These estimates hold in both
the volume-collapsing and the noncollapsed cases.

\Needspace{9\baselineskip}
\section{The global diameter bound}\label{sec:globaldiameter}

In this section, we prove the global diameter bound of
Theorem~\ref{thm:diameter}(i). The strategy is to show that the
K\"ahler potentials are uniformly H\"older continuous with respect to
the fixed metric $\omega_0$, and then to convert this into a uniform
H\"older bound for the distance functions $d_t$. We begin with the
behavior of the potentials in time.

\begin{lemma}\label{lem:timemodulus}
There exist a continuous function $\varphi_T$ on $X$ and a constant
$C>0$ such that, for all $0\le t<1$,
\begin{align}
 -C\bigl(1+|\log(1-t)|\bigr)&\le\dot\varphi(t)\le C,
       \label{eq:dotlog}\\
 \|\varphi(t)-\varphi_T\|_\infty
       &\le C(1-t)\bigl(1+|\log(1-t)|\bigr).
       \label{eq:timemodulus}
\end{align}
In particular, for every $\beta\in(0,1)$ there exists $C_\beta>0$ such
that, for all $0\le s,t<1$,
\begin{equation}\label{eq:timebeta}
 \|\varphi(t)-\varphi(s)\|_\infty\le C_\beta|t-s|^\beta.
\end{equation}
\end{lemma}
\begin{proof}
Integrating $\ddot\varphi=-R$ at each fixed point and using
\eqref{eq:scalarI}, we obtain
$\dot\varphi(t)\ge\dot\varphi(0)+C\log(1-t)$, while the upper bound
was proved in Proposition~\ref{prop:background}. Since the resulting
bound for $|\dot\varphi|$ is integrable in time and uniform in space,
$\varphi(t)$ is Cauchy in the uniform norm and converges to a
continuous function $\varphi_T$. Integrating from $t$ to $1$ then
proves \eqref{eq:timemodulus}. For
$0\le s<t<1$, the function $u\mapsto1+|\log(1-u)|$ is increasing, so
\begin{align*}
 \|\varphi(t)-\varphi(s)\|_\infty
 &\le C\int_s^t(1+|\log(1-u)|)\,du\\
 &\le C|t-s|\bigl(2+|\log|t-s||\bigr)
 \le C_\beta|t-s|^\beta.
\end{align*}
The case $s=t$ is immediate, and the estimate for $s>t$ follows by symmetry.
\end{proof}

Next, we show that the metrics along the flow degenerate at most
polynomially with respect to the background metric $\omega_0$.

\begin{proposition}\label{prop:polynomial}
There exist $N\ge1$ and $C>0$ such that
\begin{equation}\label{eq:polynomial}
 \tr_{\omega_0}\omega(t)\le C(1-t)^{-N},\qquad 0\le t<1.
\end{equation}
\end{proposition}
\begin{proof}
Let $S=\tr_{\omega_0}\omega$. The parabolic trace estimate with respect to
the fixed smooth metric $\omega_0$ gives
\[
 \square\log S\le C_0\tr_\omega\omega_0.
\]
Choose $A>C_0+1$ and set
\[
 Q=\log S+A(t\dot\varphi-\varphi).
\]
By \eqref{eq:Qidentity},
\begin{equation}\label{eq:traceQ}
 \square Q\le-(A-C_0)\tr_\omega\omega_0+An.
\end{equation}
We apply the maximum principle to $Q$ on $X\times[0,t_1]$, for
arbitrary $t_1<1$. At a maximum of $Q$ occurring at positive time,
\eqref{eq:traceQ} bounds $\tr_\omega\omega_0$ uniformly.
If $\lambda_1,\ldots,\lambda_n$ are the eigenvalues of $\omega$ relative
to $\omega_0$, then
\[
 S\le n\left(\sum_j\lambda_j^{-1}\right)^{n-1}
                  \prod_j\lambda_j.
\]
Moreover,
$\prod_j\lambda_j=e^{\dot\varphi}\Omega/\omega_0^n\le C$.
Thus $S\le C$ at such a maximum, and since
$t\dot\varphi-\varphi\le C$, it follows that $Q\le C$ there. The
value of $Q$ at time zero is clearly bounded as well. Since all
constants are independent of $t_1$, we conclude that
\[
 \log S\le C+A\varphi-At\dot\varphi
            \le C+N|\log(1-t)|
\]
by \eqref{eq:dotlog}. This proves \eqref{eq:polynomial}, with an
exponent $N$ depending only on the initial data.
\end{proof}

Interpolating between this polynomial bound in space and the modulus
of continuity in time, we obtain a uniform H\"older estimate for the
potentials.

\begin{proposition}\label{prop:holder}
There exist $\gamma\in(0,1)$ and $C>0$ such that
\begin{equation}\label{eq:holderpotential}
 |\varphi(p,t)-\varphi(q,t)|\le C d_0(p,q)^\gamma
 \qquad(p,q\in X,\ 0\le t<1).
\end{equation}
The same estimate holds for $\varphi_T$.
\end{proposition}
\begin{proof}
The Laplacian with respect to the fixed background metric satisfies
\[
 \Delta_{\omega_0}\varphi
       =\tr_{\omega_0}\omega(t)-\tr_{\omega_0}\omega_t,
 \qquad
 \|\Delta_{\omega_0}\varphi(t)\|_\infty\le C(1-t)^{-N}.
\]
Global elliptic $W^{2,p}$ estimates for $p>2n$, followed by Sobolev
embedding, therefore give
\begin{equation}\label{eq:backgroundgradient}
 \|d\varphi(t)\|_{g_0,\infty}\le C(1-t)^{-N}.
\end{equation}
To combine this spatial estimate with the temporal modulus of
continuity, fix $\beta=1/2$ and put $d=d_0(p,q)$. Since pairs with
$d\ge d_*$ are controlled by the uniform $C^0$ bound, it suffices to
consider $0<d<d_*$, where $d_*>0$ is a sufficiently small fixed
constant. Set
\[
 \delta=d^{1/(N+\beta)},\qquad s=\min\{t,1-\delta\}.
\]
If $s=t$, \eqref{eq:backgroundgradient} bounds the difference by
$C\delta^{-N}d$. Otherwise \eqref{eq:timebeta} and
\eqref{eq:backgroundgradient} give
\[
 |\varphi(p,t)-\varphi(q,t)|
 \le2\|\varphi(t)-\varphi(s)\|_\infty
           +|\varphi(p,s)-\varphi(q,s)|
 \le C\delta^\beta+C\delta^{-N}d.
\]
With this choice of $\delta$, both terms are bounded by $Cd^\gamma$, where
$\gamma=\beta/(N+\beta)>0$. Finally, the assertion for $\varphi_T$
follows by passing to the uniform limit.
\end{proof}

We can now convert the H\"older bound for the potentials into a
H\"older bound for the distances, with a constant uniform in time.

\begin{theorem}\label{thm:distanceholder}
Let $\gamma$ be the exponent in Proposition~\ref{prop:holder}. There
exists a constant $C>0$ such that, for all $p,q\in X$,
\begin{equation}\label{eq:distanceholder}
 d_t(p,q)\le C d_0(p,q)^{\gamma/2},
 \qquad\diam(X,g(t))\le C,
 \qquad0\le t<1.
\end{equation}
\end{theorem}
\begin{proof}
To obtain a constant uniform in time, we compare $\omega(t)$ with
auxiliary K\"ahler metrics in the fixed class $[\omega_0]$ and apply
the geometric H\"older estimate to this family.

Choose $L\ge1$ such that $\theta\le L\omega_0$, and fix $B=L+1$.
Then, for every $0\le t<1$,
\[
 \omega_t=(1-t)\omega_0+t\theta\le L\omega_0,
 \qquad B\omega_0-\omega_t\ge\omega_0.
\]
Set
\[
 u_t=\frac{\varphi(t)-\sup_X\varphi(t)}{B},\qquad
 \widetilde\omega(t)=\omega_0+\ddbar u_t
       =\frac{\omega(t)+B\omega_0-\omega_t}{B}.
\]
Thus $\widetilde\omega(t)>0$, $[\widetilde\omega(t)]=[\omega_0]$,
and $\omega(t)\le B\widetilde\omega(t)$. In particular,
$u_t\in\operatorname{PSH}(X,\omega_0)\cap C^\infty(X)$ and $\sup_Xu_t=0$.
By Propositions~\ref{prop:background} and~\ref{prop:holder},
\[
 M:=\sup_{0\le t<1}\|\varphi(t)\|_{C^\gamma(X,\omega_0)}<\infty,
 \qquad \|u_t\|_{C^\gamma(X,\omega_0)}\le\frac{2M}{B}.
\]
The distance estimate of Li for smooth K\"ahler metrics, as stated in
\cite[Theorem~2]{GSS}, therefore gives a constant
$C_*=C_*(X,\omega_0,\gamma,2M/B)$ such that
\[
 d_{\widetilde g(t)}(p,q)\le C_*d_0(p,q)^{\gamma/2}
       \qquad(0\le t<1),
 \qquad \widetilde g(t)=2\widetilde\omega(t)(\cdot,J\cdot).
\]
We emphasize that the background metric and the cohomology class of
$\widetilde\omega(t)$ are fixed, and that the normalized potentials
$u_t$ have a uniform $C^\gamma$ bound. Thus every parameter controlling
$C_*$ is independent of $t$, even though $[\omega_t]$ degenerates.
Note also that the cited estimate applies to individual K\"ahler
metrics and their potentials, so that the auxiliary family need not
satisfy any evolution equation.

Finally, the inequality $g(t)\le B\widetilde g(t)$ gives a
corresponding comparison of the lengths of piecewise smooth curves,
and taking the infimum over curves joining $p$ to $q$ yields
\[
 d_t(p,q)\le\sqrt B\,d_{\widetilde g(t)}(p,q)
          \le\sqrt B\,C_*d_0(p,q)^{\gamma/2}.
\]
Since $B$ is independent of $t$, this proves the distance estimate,
and the diameter bound follows from the compactness of $(X,g_0)$.
Reversing the time normalization, we obtain
Theorem~\ref{thm:diameter}(i).
\end{proof}

\section{The Type I diameter bound for fibres of contraction}\label{sec:fibres}

We continue to assume $T=1$, and all distances in this section are
measured in the ambient space $X$.

Our starting point is the control of the K\"ahler potentials
established in Section~\ref{sec:globaldiameter} as a consequence of
the scalar curvature estimate. Namely, there are $N\ge1$ and
$\gamma\in(0,1)$ such that
\begin{equation}\label{eq:introchain}
 |\dot\varphi(t)|\le C\bigl(1+|\log(1-t)|\bigr),\qquad
 \omega(t)\le C(1-t)^{-N}\omega_0,\qquad
 \sup_{t<1}\|\varphi(t)\|_{C^\gamma(X,\omega_0)}\le C.
\end{equation}
Recall that the trace bound was obtained by applying the maximum
principle to $\log\tr_{\omega_0}\omega+A(t\dot\varphi-\varphi)$, and
that interpolation with the temporal modulus of continuity yields the
spatial H\"older bound. The geometric H\"older estimate of Guo, Song
and Sturm \cite{GSS}, whose smooth case originates in \cite{Li}, then
gives
\begin{equation}\label{eq:introholderdistance}
 d_t(p,q)\le C d_0(p,q)^{\gamma/2}.
\end{equation}
As explained in the proof of Theorem~\ref{thm:distanceholder}, the
comparison with a fixed class makes the constant uniform, despite the
variation of the cohomology class.

The proof of Theorem~\ref{thm:diameter}(ii) rests on two further
ingredients. The first is a comparison of distances at different
times, which shows that distances can decrease forward in time at most
at the parabolic rate. The second is the vanishing in the limit of the
energy of holomorphic maps whose images lie in a fibre. Together they
show that every fibre collapses, and the comparison of distances then
propagates this collapse back to earlier times at the Type I rate.

\subsection{}
We first show that, at the parabolic scale, the conjugate heat kernel
based at a point remains concentrated near that point itself.

\begin{lemma}\label{lem:moment}
There exists a constant $C>0$ such that, for every $p\in X$ and
$0\le s<t<1$,
\begin{equation}\label{eq:moment}
 \int_X d_s(p,z)^2\,d\nu_{p,t;s}(z)\le C(t-s).
\end{equation}
The constant is independent of $p,s,t$.
\end{lemma}
\begin{proof}
Set $h=t-s$ and $\lambda=1-t$. In the definition of the reduced
distance, we test with the constant curve $\zeta(r)=p$. By
\eqref{eq:scalarI},
\[
 \ell_{p,t}(p,s)
 \le\frac{C}{2\sqrt h}\int_0^h\frac{\sqrt r}{\lambda+r}\,dr
 \le\frac{C}{2\sqrt h}\int_0^h r^{-1/2}\,dr\le C.
\]
Note that this constant is independent of the ratio $h/\lambda$.
Perelman's reduced-density comparison
\cite[Corollary~9.5]{Perelman} consequently gives
\begin{equation}\label{eq:diagonallower}
 K(p,t;p,s)\ge(4\pi h)^{-n}e^{-\ell_{p,t}(p,s)}\ge c h^{-n}.
\end{equation}
Choose an $H_{2n}$-center $z_s$ of $(p,t)$ on the slice at time $s$.
Evaluating \eqref{eq:heatupper} at $p$ and combining it with
\eqref{eq:diagonallower}, we obtain
\[
 c h^{-n}\le C_1h^{-n}
       \exp\left(-\frac{d_s(z_s,p)^2}{C_2h}\right).
\]
After cancelling $h^{-n}$ and enlarging $C_1$ if necessary, taking
logarithms yields
\[
 d_s(z_s,p)^2\le C_2h\log(C_1/c)\le Ch.
\]
For every $z\in X$, the triangle inequality gives
$d_s(p,z)^2\le2d_s(p,z_s)^2+2d_s(z_s,z)^2$. Integrating against
$\nu_{p,t;s}$, and using its unit mass together with the defining
estimate for the $H_{2n}$-center, we obtain
\begin{align*}
 \int_Xd_s(p,z)^2\,d\nu_{p,t;s}(z)
 &\le2d_s(p,z_s)^2+2\int_Xd_s(z_s,z)^2\,d\nu_{p,t;s}(z)\\
 &\le Ch+2H_{2n}h.
\end{align*}
This proves the lemma.
\end{proof}

As a consequence, distances can decrease forward in time at most at
the parabolic rate.

\begin{lemma}\label{lem:expansion}
There exists a constant $C>0$ such that, for every $p,q\in X$ and
$0\le s<t<1$,
\begin{equation}\label{eq:expansion}
 d_s(p,q)\le d_t(p,q)+C\sqrt{t-s}.
\end{equation}
\end{lemma}
\begin{proof}
Put $\mu_p=\nu_{p,t;s}$ and $\mu_q=\nu_{q,t;s}$. The only coupling of
$\delta_p$ with a probability measure $\mu$ is $\delta_p\otimes\mu$.
Therefore
\[
 d_{W_1}^{g(s)}(\delta_p,\delta_q)=d_s(p,q),
\]
\begin{align*}
 d_{W_1}^{g(s)}(\delta_p,\mu_p)
 &=\int_Xd_s(p,z)\,d\mu_p(z)\\
 &\le\left(\int_Xd_s(p,z)^2\,d\mu_p(z)\right)^{1/2}
 \le C\sqrt{t-s},
\end{align*}
where the inequalities follow from the Cauchy-Schwarz inequality and
Lemma~\ref{lem:moment}. The same bound holds with $q$ in place of $p$.
Applying the triangle inequality for $d_{W_1}^{g(s)}$, followed by
\eqref{eq:Wcontract}, we obtain
\begin{align*}
 d_s(p,q)
 &\le d_{W_1}^{g(s)}(\delta_p,\mu_p)
       +d_{W_1}^{g(s)}(\mu_p,\mu_q)
       +d_{W_1}^{g(s)}(\mu_q,\delta_q)\\
 &\le d_t(p,q)+C\sqrt{t-s}.
\end{align*}
\end{proof}

\subsection{}
We next turn to holomorphic maps into the fibres of $\Phi$. To pass
from their energy to distances, we need the following Poincar\'e-type
inequality for maps, with a constant independent of the target.

\begin{lemma}\label{lem:targetPoincare}
Let $(Z,h)$ be a fixed compact   Riemannian manifold, and let
\[
 d\sigma=\frac{dV_h}{\Vol_h(Z)}
\]
be its probability volume measure. Then there exists $C=C(Z,h)>0$ such
that, for every smooth map $F:Z\to(M,g_M)$ into a Riemannian manifold,
\begin{equation}\label{eq:targetPoincare}
 \int_{Z\times Z}d_{g_M}(F(z),F(w))^2\,d\sigma(z)d\sigma(w)
       \le C\int_Z|dF|_{h,g_M}^2\,d\sigma.
\end{equation}
The constant depends only on the source metric. In particular, it is
independent of the target and of the map $F$.
\end{lemma}
\begin{proof}
The assertion is immediate when $Z$ is a point. Otherwise, choose
finitely many coordinate balls $U_1,\ldots,U_L$ covering $Z$, each with
convex Euclidean image and uniformly comparable coordinate metric and
volume density. Write $E=|dF|_{h,g_M}^2$.
Within a single chart, we compose $F$ with the straight segment
joining $z$ to $w$ in coordinates. Estimating the length of the
resulting curve and applying the Cauchy-Schwarz inequality gives
\[
 d_{g_M}(F(z),F(w))^2
 \le C|z-w|^2\int_0^1E((1-r)z+rw)\,dr.
\]
To integrate this estimate, split the $r$-integral at $1/2$. On the
first half fix $w$ and substitute $v=(1-r)z+rw$. Its inverse Jacobian
is at most $2^{\dim_{\mathbb R}Z}$. On the second half, we fix $z$ and
substitute $v=(1-r)z+rw$ instead. Convexity ensures that $v$ remains
in the chart. Together with the uniform comparison of the coordinate
metrics and volume densities, this gives
\[
 \int_{U_i\times U_i}d_{g_M}(F(z),F(w))^2\,d\sigma(z)d\sigma(w)
 \le C_i\int_{U_i}E\,d\sigma.
\]
To pass from local to global, note that the intersection graph of the
cover is connected, since $Z$ is connected. Given any two endpoint
charts, we join them by a chain of overlapping charts without
repetition, and choose an intermediate point in each successive
overlap. The triangle inequality in the target and
$(\sum_{j=1}^{\ell}b_j)^2\le\ell\sum_jb_j^2$ bound the squared
distance between the endpoints by the sum of the squared distances
between successive points, each pair of which lies in a single chart.
We then integrate the endpoints over their charts, and average each
intermediate point with respect to the probability measure obtained
by restricting $\sigma$ to its overlap. The preceding local estimate
bounds every resulting term by a constant times $\int_ZE\,d\sigma$,
the denominators introduced by the averaging being the positive
measures of fixed nonempty open overlaps. Summing over the finitely
many pairs of endpoint charts proves the estimate, with a constant
depending only on the source metric.
\end{proof}

\begin{lemma}\label{lem:fibreenergy}
Let $\omega(t)=\omega_t+\ddbar\varphi(t)$ be the K\"ahler-Ricci flow
of Section~\ref{sec:background}, normalized so that $T=1$, with
Riemannian distance $d_t$. Let $f:Z\to X$ be a holomorphic map from a
compact   K\"ahler manifold $Z$ such that $\Phi\circ f$ is
constant. Fix a K\"ahler form $\eta$ on $Z$, set $k=\dim_{\mathbb C}Z$,
and equip $Z$ with the metric $h_\eta=2\eta(\cdot,J\cdot)$ and the
probability volume measure
\[
 d\sigma=\frac{\eta^k}{\int_Z\eta^k}.
\]
Then there exists a constant $C_{Z,f,\eta}>0$, independent of $t$, such that
\begin{equation}\label{eq:fibreenergy}
 \int_{Z\times Z}d_t(f(z),f(w))^2\,d\sigma(z)d\sigma(w)
             \le C_{Z,f,\eta}(1-t),\qquad 0\le t<1.
\end{equation}
\end{lemma}
\begin{proof}
If $k=0$, the assertion is immediate. Otherwise, since
$\theta=\Phi^*\theta_Y$ and $\Phi\circ f$ is constant, we have
$f^*\theta=0$, and pulling back the potential expression for the flow
gives
\[
 f^*\omega(t)=(1-t)f^*\omega_0+\ddbar(\varphi(t)\circ f).
\]
Consequently, the trace-wedge identity and Stokes' theorem imply
\begin{equation}\label{eq:energyclass}
 \int_Z\tr_\eta f^*\omega(t)\,\eta^k
 =k\int_Z f^*\omega(t)\wedge\eta^{k-1}
 =k(1-t)\int_Z f^*\omega_0\wedge\eta^{k-1}.
\end{equation}
Since $f$ is holomorphic, the real Hilbert-Schmidt norm of $df$
satisfies $|df|_{h_\eta,g(t)}^2=2\tr_\eta f^*\omega(t)$ in our metric conventions. Thus its energy is
\[
 \int_Z|df|_{h_\eta,g(t)}^2\,d\sigma
 =\frac{2k(1-t)}{\int_Z\eta^k}
       \int_Z f^*\omega_0\wedge\eta^{k-1}.
\]
Applying Lemma~\ref{lem:targetPoincare} to the fixed source
$(Z,h_\eta)$ and the target $(X,g(t))$, and recalling that its
constant is independent of the target metric, we obtain
\eqref{eq:fibreenergy}.
\end{proof}

We can now prove the Type I diameter bound for all fibres.

\begin{proposition}\label{prop:allfibres}
There exists a constant $C>0$, independent of $y$ and $s$, such that
\[
 \diam_{d_s}\Phi^{-1}(y)\le C\sqrt{1-s}
       \qquad(y\in Y,\ 0\le s<1).
\]
\end{proposition}
\begin{proof}
Let $C_0$ be the uniform constant in \eqref{eq:expansion}.
We first prove the bound for an irreducible component
$F_0$ of a reduced fibre. Since a zero-dimensional component has zero
diameter, we may assume that $F_0$ is positive-dimensional. Embedded
analytic resolution \cite{BM} then gives a surjective holomorphic map
$f:Z\to F_0$ with $Z$ compact, connected, and K\"ahler. Indeed, we may
resolve $F_0$ by blow-ups of the ambient compact K\"ahler manifold $X$
along smooth centers and take its smooth strict transform, since these
blow-ups are K\"ahler, restriction gives a K\"ahler form $\eta$ on $Z$.

Fix $s<1$ and define the continuous, nonnegative function
\[
 A_s(z,w)=\bigl(d_s(f(z),f(w))-C_0\sqrt{1-s}\bigr)_+,
 \qquad (z,w)\in Z\times Z,
\]
where $r_+=\max\{r,0\}$. For every $s<t<1$, \eqref{eq:expansion}
and $\sqrt{t-s}\le\sqrt{1-s}$ imply
$A_s(z,w)\le d_t(f(z),f(w))$. Lemma~\ref{lem:fibreenergy} therefore gives
\[
 0\le\int_{Z\times Z}A_s^2\,d\sigma\,d\sigma
 \le\int_{Z\times Z}d_t(f(z),f(w))^2\,d\sigma\,d\sigma
 \le C_{Z,f,\eta}(1-t).
\]
Since the integral of $A_s^2$ is independent of $t$, letting
$t\rightarrow1$ shows that it vanishes. As $\sigma\otimes\sigma$ has
full support and $A_s$ is continuous, it follows that $A_s\equiv0$,
for a positive value would give a positive integral over an open
neighborhood. By the surjectivity of $f$, we conclude that
\begin{equation}\label{eq:componentdiameter}
 \diam_{d_s}F_0\le C_0\sqrt{1-s}.
\end{equation}
Note that the constant in this bound is independent of the resolution.

Now fix $y$. The reduced fibre is connected and has finitely many
irreducible components, say $r(y)$ of them, so its component
intersection graph is connected. Joining two points in the fibre by a
chain of components and choosing points in successive intersections,
we obtain from the triangle inequality and
\eqref{eq:componentdiameter}, now at time $t$,
\[
 \diam_{d_t}\Phi^{-1}(y)\le r(y)C_0\sqrt{1-t}\longrightarrow0
       \qquad(t\rightarrow1).
\]
This shows that each fibre collapses, although not yet uniformly,
since no uniform bound for $r(y)$ is available.

To obtain a uniform bound, fix $p,q\in\Phi^{-1}(y)$ and $s<1$, and
apply \eqref{eq:expansion} directly to this pair. Letting
$t\rightarrow1$ and using the collapse just proved, we find
\[
 d_s(p,q)\le\lim_{t\rightarrow1}\bigl(d_t(p,q)+C_0\sqrt{t-s}\bigr)
       =C_0\sqrt{1-s}.
\]
Taking the supremum over the fibre proves the proposition with a
constant independent of $y$.
\end{proof}

Since reversing the time normalization multiplies distances by
$\sqrt T$ and turns $\sqrt{1-s}$ into $\sqrt{T-t}$,
Proposition~\ref{prop:allfibres} proves Theorem~\ref{thm:diameter}(ii).

\section{Gromov-Hausdorff convergence}\label{sec:GH}

In this section, we prove Theorem~\ref{thm:GH}. The idea is to show
first that the distance functions converge to a pseudometric on $X$,
and then to identify its zero set with the fibres of $\Phi$.

\begin{lemma}\label{lem:distancelimit}
For every $p,q\in X$, the limit
\begin{equation}\label{eq:pseudometric}
 d_\infty(p,q)=\lim_{t\rightarrow1}d_t(p,q)
\end{equation}
exists. Moreover, $d_\infty$ is a continuous pseudometric on $X$, and
$d_t\to d_\infty$ uniformly on $X\times X$.
\end{lemma}
\begin{proof}
Fix $p,q\in X$. The uniform diameter bound implies that
$L=\limsup_{t\rightarrow1}d_t(p,q)$ is finite. Choose a sequence
$s_j\rightarrow1$ such that $d_{s_j}(p,q)\to L$.
For $t>s_j$, Lemma~\ref{lem:expansion} implies
\[
 d_t(p,q)\ge d_{s_j}(p,q)-C\sqrt{1-s_j}.
\]
Taking the lower limit as $t\rightarrow1$ and then letting
$j\to\infty$ shows that the lower and upper limits agree.

The triangle inequality and \eqref{eq:distanceholder} give
\begin{equation}\label{eq:pairequicont}
 |d_t(p,q)-d_t(p',q')|
 \le C d_0(p,p')^{\gamma/2}+C d_0(q,q')^{\gamma/2}.
\end{equation}
Together with the compactness of $X\times X$, this uniform modulus of
continuity implies that the pointwise convergence is uniform. Symmetry
and the triangle inequality pass to the limit, and
\eqref{eq:pairequicont} also shows that $d_\infty$ is continuous.
\end{proof}

We next identify the pairs of points at zero $d_\infty$-distance. Here
the Schwarz estimate and the fibre diameter bound play complementary
roles.

\begin{corollary} \label{cor:separation}
There is a metric $D_Y$ inducing the complex-analytic topology of $Y$
and a constant $C>0$ such that
\begin{equation}\label{eq:separation}
 D_Y(\Phi(p),\Phi(q))\le C d_t(p,q)
       \qquad(p,q\in X,\ 0\le t<1).
\end{equation}
Consequently,
\begin{equation}\label{eq:fibrerelation}
 d_\infty(p,q)=0\quad\Longleftrightarrow\quad\Phi(p)=\Phi(q).
\end{equation}
\end{corollary}
\begin{proof}
Using finitely many local ambient coordinates multiplied by smooth
partition-of-unity cutoffs, together with the cutoffs themselves, we
construct a continuous injective map $F:Y\to\mathbb R^M$ whose
components have bounded ambient derivatives. Then
$D_Y(y,y')=|F(y)-F(y')|$ induces the topology of the compact space
$Y$, and the Schwarz estimate gives
\[
 \Phi^*\theta_Y\le C\omega(t),\qquad
 |d(F\circ\Phi)|_{g(t)}\le C.
\]
Integrating along curves in $X$ proves \eqref{eq:separation}, even
when $Y$ is singular. Passing to the limit, we see that
$d_\infty(p,q)=0$ forces $\Phi(p)=\Phi(q)$, and the converse follows
from the fibre diameter bound of Proposition~\ref{prop:allfibres}.
\end{proof}

\Needspace{10\baselineskip}
\begin{proposition}\label{prop:GHall}
There exists a metric $d_T$ on $Y$ inducing its complex-analytic
topology such that, as $t\rightarrow1$,
\begin{equation}\label{eq:uniformdistance}
 \varepsilon(t):=\sup_{p,q\in X}
 |d_t(p,q)-d_T(\Phi(p),\Phi(q))|\longrightarrow0.
\end{equation}
Moreover,
\begin{equation}\label{eq:GHdistortion}
 d_{\GH}((X,d_t),(Y,d_T))\le\tfrac12\varepsilon(t).
\end{equation}
\end{proposition}
\begin{proof}
For $y,y'\in Y$, choose $p\in\Phi^{-1}(y)$ and
$q\in\Phi^{-1}(y')$, and define $d_T(y,y')=d_\infty(p,q)$.
The triangle inequality and \eqref{eq:fibrerelation} show that this
definition is independent of the chosen lifts and defines a metric
on $Y$. By \eqref{eq:pairequicont}, the map $\Phi:X\to(Y,d_T)$ is
continuous. Since the original map $\Phi:X\to Y$ is a closed
surjection, the identity from $Y$ with its complex-analytic topology
to $(Y,d_T)$ is continuous, and hence a homeomorphism, since its
domain is compact and its target is Hausdorff.

Uniform convergence in Lemma~\ref{lem:distancelimit} proves
\eqref{eq:uniformdistance}. The graph of $\Phi$ is a surjective
correspondence between $X$ and $Y$, with distortion
$\varepsilon(t)$, which proves \eqref{eq:GHdistortion}. In particular,
$(Y,d_T)$ is a compact length space, since it is a Gromov-Hausdorff
limit of compact length spaces.
\end{proof}

Proposition~\ref{prop:GHall} proves the Gromov-Hausdorff convergence
and the topological assertions of Theorem~\ref{thm:GH}, in both the
collapsed and the noncollapsed cases. The limit metric is the unique
metric on $Y$ whose pullback by $\Phi$ equals $d_\infty$, and
\eqref{eq:uniformdistance} also gives the uniform convergence of
distances stated in Section~\ref{subsec:introGH}. It remains to
establish the more precise description of the limit in the
noncollapsed case.

\bigskip

\noindent{\it Further results in the noncollapsed case.}~
On each compact subset of $X^\circ$, the Schwarz estimate and the
upper bound for the volume form imply that the metrics are uniformly
equivalent to the background metric, and interior parabolic
regularity then yields smooth convergence to $\Phi^*g_T$.

With our normalization $T=1$, the gradient estimate for $U$, its
uniform limit $U\to\varphi_T+n$, and fibre collapse show that
$\varphi_T$ descends to a continuous potential $\psi_T$ on $Y$.
The classical \L{}ojasiewicz distance inequality
\cite[Theorem~6.4 and Remark~6.5]{BMsubanalytic}, applied locally to the defining equations
of $X\times_YX$, together with the H\"older bound on $\varphi_T$,
gives ambient H\"older continuity of $\psi_T$.

Set $S=Y\setminus Y^\circ$. Since $Y$ is normal and connected, its
regular locus is connected, and so is $Y^\circ$, as $S$ is a proper
analytic subset. In particular, $d_T^\circ$ is finite. The local
geometric H\"older estimate
\cite[Section~4 and Section~4.1]{GSS}, with the excluded analytic set
in each finite-projection chart enlarged to contain the image of $S$,
gives $d_T^\circ\le C D_Y^\kappa$, where $D_Y$ is the background
metric in Corollary~\ref{cor:separation}. The Schwarz estimate gives
$D_Y\le C d_T$, hence $d_T^\circ\le C d_T^\kappa$.
Finally, smooth convergence along lifted curves gives
$d_T\le d_T^\circ$.

\section{Partial \texorpdfstring{$C^0$}{C0}-estimates}\label{sec:partialC0}

In this section, we establish the local partial $C^0$-estimate, in
the sense of Tian \cite{Tian90}, by adapting the estimates for
holomorphic sections and the peak-section
construction of \cite[Section~8]{JST}. Two features of the present
setting require some care. First, since no rationality assumption is
made, the polarization is only defined on a neighborhood of the
limiting fibre, and is constructed from the potential of Part I.
Second, the gradient bound for a Ricci potential used in \cite{JST} is
replaced by a short-time correction, and the scalar curvature estimate
of Part I supplies the analytic estimates at the prescribed base
point. Throughout, we use $\hat g$ and $\hat h$ for the scaled metric
and bundle metric, and $\tilde h$ for the weighted bundle metric.

\subsection{}
Fix $x\in X$, let $t_i\rightarrow1$, and put
\[
 \tau_i=1-t_i,\qquad p_i=x,\qquad
 g_{i,t}=\tau_i^{-1}g(1+\tau_i t),\qquad
 \omega_{i,t}=\tau_i^{-1}\omega(1+\tau_i t).
\]
We shall work on the slice $t=-1$, so that
$g_{i,-1}=\tau_i^{-1}g(t_i)$. By the results of
Section~\ref{sec:soliton}, after passing to a subsequence, these
slices converge in the pointed Gromov-Hausdorff topology to a complete
singular K\"ahler-Ricci shrinker, smoothly on its regular locus, and
its metric tangent cones are Ricci-flat K\"ahler cones, again with
smooth convergence on regular compact subsets. We stress that these
geometric results are established independently of the section
estimates below (see also \cite[Theorem~B]{HZ}).

Write $y=\Phi(x)$ and choose Stein neighborhoods
\[
 y\in V_0\Subset V_1\subset Y,
 \qquad D_j=\Phi^{-1}(V_j),
\]
inside an ambient chart on which $\theta_Y=\ddbar\rho$.
The domains $D_j$ have proper smooth plurisubharmonic exhaustions
obtained by pullback from $V_j$, and $\overline{D_0}\subset D_1$.
If $Y$ is a point, take $D_0=D_1=X$ and $\rho=0$.
Set
\[
 L_i=K_X^{-1}|_{D_1},\qquad
 w_i=\frac{U(t_i)+\rho\circ\Phi}{\tau_i}.
\]
Note that the underlying line bundle is independent of $i$. Let
$h_{K^{-1},i}$ denote its metric induced by $\omega(t_i)$, and define
\begin{equation}\label{eq:localbundlemetric}
 h_i=h_{K^{-1},i}e^{-w_i}.
\end{equation}
We write $\Theta(h)=-\ddbar\log h$ for the curvature of a bundle
metric. By \eqref{eq:Ucurv},
\begin{equation}\label{eq:exactpolarization}
 \Theta(h_i)=\Ric(\omega(t_i))+\ddbar w_i=\omega_{i,-1}.
\end{equation}
Thus $h_i$ defines a local polarization of the rescaled K\"ahler
metric on $D_1$, without any rationality assumption on the initial
class.

The Schwarz estimate implies that, for every fixed radius $A>0$,
\begin{equation}\label{eq:ballsindomain}
 B_{g_{i,-1}}(p_i,A)\subset D_0
       \qquad\hbox{for all sufficiently large }i.
\end{equation}
Indeed, $d_{g(t_i)}(x,X\setminus D_0)$ has a fixed positive lower
bound, while distances are multiplied by $\tau_i^{-1/2}$ under
rescaling. The same argument gives
\[
 d_{g_{i,-1}}(D_0,X\setminus D_1)\longrightarrow\infty.
\]
When $D_1=X$, there is of course no boundary restriction.

For $k\in\mathbb N$, we use the scaled notation
\[
 \hat g=k g_{i,-1},\qquad
 \hat\omega=k\omega_{i,-1},\qquad
 \hat h=h_i^k,\qquad \Theta(\hat h)=\hat\omega.
\]
We retain the convention $\hat g=2\hat\omega(\cdot,J\cdot)$ of
Section~\ref{sec:background}. In the interior estimates, $\nabla$
is the Chern connection of $\hat h$, with its tensor norms measured
using $\hat g$, $dV=dV_{\hat g}$, $\Delta=\Delta_{\hat g}$, and
$B(q,R)=B_{\hat g}(q,R)$. In particular,
\[
 |\sigma|^2=|\sigma|^2_{\hat h},\qquad
 |\nabla\sigma|^2=|\nabla\sigma|^2_{\hat h,\hat g}.
\]
All constants in these scaled estimates are independent of $i$ and
$k$. For the statement of the final result, we define the Bergman
density at the original blow-up scale:
\begin{equation}\label{eq:localbergman}
 \rho^D_{i,k}(q)=
 \sup_{\substack{\sigma\in H^0(D,L_i^k)\cap L^2(D,h_i^k)\\\sigma\ne0}}
 \frac{|\sigma(q)|_{h_i^k}^2}
      {\int_D|\sigma|_{h_i^k}^2dV_{g_{i,-1}}},\qquad D\subset D_1.
\end{equation}
By the interior estimates below, point evaluation is a bounded linear
functional on this space of sections, so this density is well defined.

\subsection{}
The integration-by-parts argument in \cite[Proposition~8.4]{JST}
relies on a gradient bound for a Ricci potential. In the present
setting, we replace it by a short-time correction, which satisfies the
following quadratic-form estimate.

\begin{lemma}\label{lem:shortpotential}
Let $\omega(u)$ be a smooth K\"ahler-Ricci flow on a compact K\"ahler
manifold $X$ for $-1\le u\le0$, satisfying $|R|\le\Lambda$.
Set $\omega=\omega(0)$.
Then for every $0<\delta\le1$ there exists a smooth function
$v_\delta$ such that
\begin{align}
 \|v_\delta\|_\infty&\le\Lambda\delta/2,\label{eq:vsmall}\\
 \Ric(\omega)+\ddbar v_\delta
 &=\delta^{-1}\bigl(\omega(-\delta)-\omega\bigr)
       \ge-\delta^{-1}\omega,\label{eq:vtwist}\\
 \Delta_\omega v_\delta&\ge-n/\delta-\Lambda.\label{eq:vlap}
\end{align}
Moreover, for every compactly supported smooth real function $F$,
\begin{equation}\label{eq:vform}
 \int |dv_\delta|^2F^2dV
 \le C\delta^2\int|dF|^2dV+C\int F^2dV,
\end{equation}
where all norms and integrals are taken at time zero, and $C$ depends
only on $n,\Lambda$.
\end{lemma}
\begin{proof}
Set
\[
 \ell(u)=\log\frac{\omega(u)^n}{\omega(-\delta)^n},\qquad
 v_\delta=\frac1\delta\int_{-\delta}^0
                    \bigl(\ell(0)-\ell(u)\bigr)\,du.
\]
Since $\ell_u=-R$, changing the order of integration gives
\[
 v_\delta=-\frac1\delta\int_{-\delta}^0(u+\delta)R(u)\,du,
 \qquad
 |v_\delta|\le\frac\Lambda\delta
                 \int_{-\delta}^0(u+\delta)\,du
              =\frac{\Lambda\delta}{2}.
\]
This proves \eqref{eq:vsmall}. Moreover,
\[
 \Ric(\omega(u))-\Ric(\omega(0))
       =\ddbar\bigl(\ell(0)-\ell(u)\bigr).
\]
Averaging this identity and integrating the flow equation yields
\[
 \Ric(\omega)+\ddbar v_\delta
 =\frac1\delta\int_{-\delta}^0\Ric(\omega(u))\,du
 =\frac{\omega(-\delta)-\omega(0)}{\delta}.
\]
The positivity of $\omega(-\delta)$ proves \eqref{eq:vtwist}.
Taking the trace with respect to $\omega(0)$ and using
$R(0)\le\Lambda$ gives \eqref{eq:vlap}.

To prove the energy estimate, put
$b_\delta=v_\delta-\inf_Xv_\delta$, so that
$0\le b_\delta\le\Lambda\delta$ and
$db_\delta=dv_\delta$, $\Delta b_\delta=\Delta v_\delta$.
Set $J=\int F^2|db_\delta|^2dV$. Integration by parts gives
\[
 J=-\int b_\delta F^2\Delta b_\delta\,dV
       -2\int b_\delta F\langle db_\delta,dF\rangle\,dV.
\]
The lower bound for $\Delta b_\delta$ and Young's inequality imply
\[
 J\le\Lambda\delta(n/\delta+\Lambda)\int F^2dV
       +\tfrac12J+2\Lambda^2\delta^2\int|dF|^2dV.
\]
Absorbing $J/2$ into the left-hand side and using $\delta\le1$
proves \eqref{eq:vform}.
\end{proof}

We apply Lemma~\ref{lem:shortpotential} to
$\omega(u)=\omega_{i,-1+u}$ on $[-1,0]$, after discarding finitely
many indices, this is possible since Part I gives a uniform scalar
curvature bound $|R|\le\Lambda$. Denoting the corrections by
$v_{i,\delta}$ and writing $v_i=v_{i,1}$, we have in particular
\[
 \|v_i\|_\infty\le\Lambda/2,\qquad
 \Ric(\omega_{i,-1})+\ddbar v_i
                    =\omega_{i,-2}-\omega_{i,-1}.
\]
Thus the potential $w_i$ in \eqref{eq:localbundlemetric} defines the
polarization, while $v_i$ will provide a bounded weight for the
$L^2$ estimate. For the scaled metric $\hat g$, we apply the same
lemma to
\[
 \hat\omega_i(u)=k\omega_{i,-1+u/k},\qquad -1\le u\le0.
\]
Its scalar curvature is bounded by $\Lambda/k$, and its correction
$\widehat v_\delta=v_{i,\delta/k}$ satisfies
\eqref{eq:vsmall}-\eqref{eq:vform} in the metric $\hat g$, with
uniform constants.

\begin{proposition} \label{prop:localanalytic}
For every $0<r<R<\infty$, there exists $K=K(r,R)>0$ such that, for all
$k\ge2$ and every holomorphic section $\sigma$ of $L_i^k$ on
$B(q_0,R)\Subset D_1$,
\begin{equation}\label{eq:sectionbounds}
 \begin{aligned}
 \|\sigma\|_{L^\infty(B(q_0,r),\hat h)}
       &\le K\|\sigma\|_{L^2(B(q_0,R),\hat h,\hat g)},\\
 \|\nabla\sigma\|_{L^\infty(B(q_0,r),\hat h,\hat g)}
       &\le K\|\sigma\|_{L^2(B(q_0,R),\hat h,\hat g)}.
 \end{aligned}
\end{equation}
Here the constant $K$ is independent of $i$, $k$, and $q_0$.
\end{proposition}
\begin{proof}
By Section~\ref{sec:background}, the Sobolev inequality for $\hat g$
is uniform. Since the bundle curvature is $\hat\omega$, we may use
the Bochner formulas for sections of \cite[Lemma~8.2]{JST}, with the
dimensional constants adjusted to our metric convention. Write
$B_r=B(q_0,r)$ and $B_R=B(q_0,R)$. The Bochner identity for sections
gives
\[
 \Delta|\sigma|^2\ge-C|\sigma|^2.
\]
The Sobolev inequality and Moser iteration therefore give
\[
 \sup_{B_r}|\sigma|^2\le C(r,R)\int_{B_R}|\sigma|^2.
\]
For the derivative estimate, we adapt the argument of
\cite[Lemma~8.2 and Proposition~8.4]{JST} using \eqref{eq:vform}. Set
$a=\nabla^{1,0}\sigma$ and $q=|a|^2$, use $\nabla$ also for the full
induced covariant derivative on this bundle-valued one-form, including
its $(0,1)$-part, which is a curvature term in $\sigma$. The Bochner
formula for $q$ contains a positive term controlling $|\nabla a|^2$,
a Ricci contraction with $a\otimes\bar a$, and a zeroth-order term
bounded below by $-Cq$, since the bundle curvature is parallel, its
covariant derivatives vanish. By \eqref{eq:vtwist},
\[
 \Ric(\hat\omega)\ge-\ddbar \widehat v_\delta-\delta^{-1}\hat\omega.
\]
For a nonnegative cutoff $\eta$ and $p\ge2$, define, suppressing the
volume form,
\[
 E_p=\int\eta^2q^{p-2}|dq|^2,\qquad
 H_p=\int\eta^2q^{p-1}|\nabla a|^2,\qquad
 Q_p=\int_{\operatorname{supp}\eta}q^p.
\]
Testing the Bochner inequality against $\eta^2q^{p-1}$ and absorbing
the cutoff cross term yields
\[
 c(p-1)E_p+cH_p
 \le C(1+\delta^{-1})Q_p
       +C\int|d\eta|^2q^p+C|I_p|,
\]
where
\[
 I_p=\int\eta^2q^{p-1}
             \langle\ddbar \widehat v_\delta,a\otimes\bar a\rangle.
\]
Here the contraction is taken with respect to the K\"ahler metric and
the bundle metric, and $c,C>0$ are independent of $i,k,p,\delta$.

To estimate $I_p$, we integrate one derivative of
$\ddbar \widehat v_\delta$ by parts. This produces derivatives of
$\eta^2$, $q^{p-1}$, and $a\otimes\bar a$, the last two of which are
controlled using $|dq|\le2\sqrt q\,|\nabla a|$. Consequently,
\[
 |I_p|\le Cp\int\eta^2|d\widehat v_\delta||\nabla a|q^{p-1/2}
              +C\int\eta|d\eta||d\widehat v_\delta|q^p.
\]
Define
\[
 J_p=\int\eta^2|d\widehat v_\delta|^2q^p.
\]
Young's inequality, with a fixed sufficiently small coefficient
in front of $H_p$, gives
\[
 C|I_p|\le\tfrac c2H_p+Cp^2J_p+C\int|d\eta|^2q^p.
\]
We now apply \eqref{eq:vform} with $F=\eta q^{p/2}$:
\[
 \begin{split}
 J_p&\le C\delta^2\int|d(\eta q^{p/2})|^2+CQ_p\\
    &\le C\delta^2p^2E_p
          +C\delta^2\int|d\eta|^2q^p+CQ_p.
 \end{split}
\]
Combining these inequalities, and using $\delta\le1$, yields
\[
 c(p-1)E_p+\tfrac c2H_p
 \le C(1+\delta^{-1}+p^2)(1+\|d\eta\|_\infty^2)Q_p
       +Cp^4\delta^2E_p.
\]
The key point is now the choice
\[
 \delta=\delta_*p^{-2},
\]
where $0<\delta_*\le1$ is a fixed, sufficiently small constant. Then
$Cp^4\delta^2=C\delta_*^2\le c/2$, so that the final term is absorbed
by $c(p-1)E_p$. Since $\delta^{-1}=\delta_*^{-1}p^2$, it follows that
\[
 E_p\le Cp(1+\|d\eta\|_\infty^2)Q_p,
 \qquad
 \int|d(\eta q^{p/2})|^2
       \le Cp^3(1+\|d\eta\|_\infty^2)Q_p.
\]
We emphasize that only the auxiliary correction $\widehat v_\delta$
depends on the iteration exponent $p$, while the metric, the bundle
metric, and the section remain fixed, the choice
$\delta=\delta_*p^{-2}$ provides precisely the small coefficient
needed to absorb the gradient term uniformly throughout the iteration.
The above calculations are justified at the zeros of $q$ by replacing
$q$ with $q+\varepsilon$ and then letting $\varepsilon\rightarrow0$.
Moreover, since only first derivatives of the cutoffs are needed,
Lipschitz distance cutoffs and smooth approximation suffice.

Let $\beta>1$ be an exponent in the uniform Sobolev inequality, and
choose a cutoff with $\eta=1$ on $B_r$, supported in $B_R$, and
$|d\eta|\le C/(R-r)$. We obtain
\[
 \|q\|_{L^{p\beta}(B_r)}
 \le\bigl[Cp^3(1+(R-r)^{-2})\bigr]^{1/p}
                         \|q\|_{L^p(B_R)}.
\]
Moser iteration with $p_l=2\beta^l$ and $r_l=r+(R-r)2^{-l}$ then
gives an interior $L^2$-to-$L^\infty$ bound for $q$. To replace the
$L^2$ norm on the right-hand side by an $L^1$ norm, we use
\[
 \|q\|_{L^2(B_R)}^2
       \le\bigl(\sup_{B_R}q\bigr)\int_{B_R}q
\]
on nested interior balls, and then apply Young's inequality and
iterate in the radius. Thus, for $r<R'<R$,
\[
 \sup_{B_r}q\le C(r,R')\int_{B_{R'}}q.
\]
Finally, the Bochner identity for sections gives
$\Delta|\sigma|^2\ge c_0q-C_0|\sigma|^2$, with fixed dimensional
constants $c_0>0$ and $C_0\ge0$. Multiplying this inequality by the
square of a cutoff equal to one on $B_{R'}$ and supported in $B_R$,
and using integration by parts, the bound
$|d|\sigma|^2|\le C|\sigma|\sqrt q$, and Young's inequality, we obtain
\[
 \int_{B_{R'}}q\le C(R',R)\int_{B_R}|\sigma|^2.
\]
Combining the last two estimates proves \eqref{eq:sectionbounds}.
\end{proof}

\begin{corollary}\label{cor:sectionD0}
For each fixed $k\ge2$, there exists $C_k>0$ such that, for all
sufficiently large $i$, every
$\sigma\in H^0(D_1,L_i^k)\cap L^2(D_1,h_i^k)$ satisfies
\begin{equation}\label{eq:sectionD0}
 \sup_{D_0}\bigl(|\sigma|_{h_i^k}
                 +|\nabla\sigma|_{h_i^k,g_{i,-1}}\bigr)
 \le C_k\|\sigma\|_{L^2(D_1,h_i^k,g_{i,-1})}.
\end{equation}
\end{corollary}
\begin{proof}
For all sufficiently large $i$, each point of $D_0$ is the center
of a ball of fixed radius contained in $D_1$. Applying
Proposition~\ref{prop:localanalytic} on these balls and rescaling
from $\hat g$ to $g_{i,-1}$, we obtain the estimate, with constants
depending on $k$.
\end{proof}

\subsection{}
Following \cite[Proposition~8.6]{JST}, we introduce the weighted metric
\[
 \tilde h=\hat h e^{-v_i}=h_i^ke^{-v_i}.
\]
This metric is used only to solve the $\bar\partial$ equation in
$L^2$, while derivatives of sections continue to be measured using
$\hat h$ and $\hat g$.

\begin{proposition}[$L^2$ estimate]\label{prop:dbar}
There exists $C>0$, independent of $i$ and $k\ge2$, such that every
smooth compactly supported $L_i^k$-valued $(0,1)$-form $\tau$ on
$D_1$ with $\bar\partial\tau=0$ admits a section $\nu$ satisfying
$\bar\partial\nu=\tau$ and
\begin{equation}\label{eq:dbar}
 \int_{D_1}|\nu|_{\hat h}^2dV
       \le C\int_{D_1}|\tau|_{\hat h,\hat g}^2dV.
\end{equation}
\end{proposition}
\begin{proof}
Since the Ricci form is unchanged by constant scaling,
\[
 \begin{split}
 \Ric(\hat\omega)+\Theta(\tilde h)
   &=\Ric(\omega_{i,-1})+k\omega_{i,-1}+\ddbar v_i\\
   &=(k-1)\omega_{i,-1}+\omega_{i,-2}
     \ge\tfrac12\hat\omega.
 \end{split}
\]
The weighted Kodaira-H\"ormander estimate therefore gives
\[
 \bar\partial\nu=\tau,\qquad
 \int_{D_1}|\nu|_{\tilde h}^2dV
       \le C\int_{D_1}|\tau|_{\tilde h,\hat g}^2dV.
\]
When $D_1\ne X$, we apply the estimate on regular sublevel sets of
the smooth proper plurisubharmonic exhaustion of $D_1$, where the
boundary Levi term is nonnegative, and pass to a weak limit. When
$D_1=X$, we apply the compact version of the estimate (see
\cite{Hormander} and \cite[Proposition~8.6]{JST}). Finally,
\[
 e^{-\Lambda/2}\hat h\le\tilde h\le e^{\Lambda/2}\hat h,
\]
so the weighted and unweighted norms are uniformly equivalent, and
\eqref{eq:dbar} follows.
\end{proof}

We are now ready to apply the Gaussian peak-section construction of
\cite[Lemma~8.7]{JST}, using the local estimates above. Since our
bundles and domains differ from those of \cite{JST}, we include the
verification of the relevant hypotheses in the present setting.

\begin{lemma}\label{lem:peakrecovery}
Fix $A>0$ and points $q_i\in B_{g_{i,-1}}(p_i,A)$. After passing
to a subsequence, there exist an integer $k\ge2$, constants $M,b>0$,
and sections $\sigma_i\in H^0(D_1,L_i^k)$ such that
\[
 \|\sigma_i\|_{L^2(D_1,\hat h,\hat g)}\le M,\qquad
 |\sigma_i(q_i)|_{\hat h}\ge b.
\]
Here the degree $k$ and the constants $M,b$ are independent of $i$.
\end{lemma}
\begin{proof}
By Proposition~\ref{prop:metricsoliton}, after passing to a
subsequence, the pointed spaces converge to $(Z,d,p_\infty)$ and $q_i$
converges to a point $q\in Z$. Rescaling at $q$ by integers
$k_j\to\infty$, which multiplies the metric by $k_j$ and distances by
$\sqrt{k_j}$, we obtain a tangent cone $C(\Sigma)$, which is
Ricci-flat and K\"ahler. We write $o$ for its vertex. The convergence
is smooth on compact subsets of the regular locus, and the curvatures
of the corresponding unitary connections on $L_i^{k_j}$ converge to
the K\"ahler form of the cone.

On the regular part of the cone, any limiting bundle differs from the
Gaussian bundle with metric $h_C=e^{-r^2/4}$ by a flat unitary bundle,
where $r=d_C(o,\cdot)$ and $\ddbar(r^2/4)=\omega_C$ in our convention.
Since the cone has log terminal singularities by
\cite[Theorem~B(iv)]{HZ}, Braun's regional fundamental-group theorem
\cite[Theorem~1]{Braun}, applied at the vertex and combined with
radial dilation, shows that the fundamental group of the regular part
of the cone is finite. Hence a fixed tensor power trivializes the flat
unitary factor, and after replacing the degrees by this multiple and
rescaling the cone, we may take the Gaussian bundle as the model.

The constant section of the Gaussian bundle has finite $L^2$ norm. We
cut it off at large radius and near the singular set. The first error
is controlled by Gaussian decay, and the second by the vanishing
$W^{1,2}$ capacity of the singular set. The resulting sections are
supported in regular compact subsets and have arbitrarily small
$\bar\partial$-error. We transport them first to $Z$ and then to the
approximating slices, whose domains contain their supports by
\eqref{eq:ballsindomain}. As in \cite[Lemma~8.7]{JST}, we first choose
the cutoffs and a sufficiently large but fixed degree $k$, and then
take $i$ large.

It remains to verify the lower bound at $q_i$. Choose a regular
comparison point $z$ so close to $o$ that the norm of the Gaussian
section at $z$ exceeds $3/4$. The transported sections $s_i$ then have
uniformly bounded $L^2(D_1,\hat h,\hat g)$ norm and arbitrarily small
$\|\bar\partial s_i\|_2$, and satisfy $|s_i(z_i)|>2/3$. By
Proposition~\ref{prop:dbar}, we can solve
\[
 \bar\partial\nu_i=\bar\partial s_i,\qquad
 \|\nu_i\|_2\le C\|\bar\partial s_i\|_2,
\]
and set $\sigma_i=s_i-\nu_i$. On a regular ball about $z_i$, the
right-hand side converges to zero in smooth norms, so the smooth
convergence of the geometry and interior elliptic estimates make
$|\nu_i(z_i)|$ arbitrarily small. Thus $|\sigma_i(z_i)|>1/2$.
Finally, Proposition~\ref{prop:localanalytic} gives a uniform bound
for $|\nabla\sigma_i|$ on a ball containing $q_i$ and $z_i$. Choosing
the comparison point close enough to the cone vertex, and the errors
small enough, that this bound times $d_{\hat g}(q_i,z_i)$ is less than
$1/4$, and integrating along a minimizing path, we obtain
$|\sigma_i(q_i)|>1/4$. This proves the lemma.
\end{proof}

We can now state and prove the local partial $C^0$-estimate.

\begin{theorem}[Local partial $C^0$ estimate]\label{thm:partialC0}
For every $A>0$, there exist $k_A\in\mathbb N$ and $c_A>0$
such that, for all sufficiently large $i$,
\begin{equation}\label{eq:partialC0}
 \rho^{D_0}_{i,k_A}(q)\ge c_A
       \qquad\hbox{for all }q\in B_{g_{i,-1}}(p_i,A).
\end{equation}
In particular, the estimate holds at $p_i=x$.
\end{theorem}
\begin{proof}
For the sections in Lemma~\ref{lem:peakrecovery}, the volume scaling
$dV_{\hat g}=k^n dV_{g_{i,-1}}$ gives
\[
 \rho^{D_1}_{i,k}(q_i)\ge \frac{k^nb^2}{M^2}>0.
\]
We claim that there exist an integer $K\ge2$ and a constant $c>0$
such that, for every sufficiently large $i$ and every
$q\in B_{g_{i,-1}}(p_i,A)$, some $2\le k\le K$ satisfies
$\rho^{D_1}_{i,k}(q)\ge c$. Otherwise, there would exist
$i_j\to\infty$ and points $q_{i_j}$ at which all densities in degrees
$2\le k\le j$ are less than $1/j$, contradicting
Lemma~\ref{lem:peakrecovery}.

Set $k_A=\operatorname{lcm}(2,\ldots,K)$. For each such $q$, choose a
degree $k$ satisfying the preceding lower bound and a section
$\sigma\in H^0(D_1,L_i^k)$ with
\[
 \int_{D_1}|\sigma|_{h_i^k}^2dV_{g_{i,-1}}=1,
 \qquad |\sigma(q)|_{h_i^k}^2\ge c/2.
\]
Let $m=k_A/k$. Corollary~\ref{cor:sectionD0} gives
\[
 \begin{split}
 \int_{D_0}|\sigma^m|_{h_i^{k_A}}^2dV_{g_{i,-1}}
   &\le\bigl(\sup_{D_0}|\sigma|_{h_i^k}\bigr)^{2m-2}
                 \int_{D_0}|\sigma|_{h_i^k}^2dV_{g_{i,-1}}\\
   &\le C_k^{2m-2},
 \qquad |\sigma^m(q)|_{h_i^{k_A}}^2\ge(c/2)^m.
 \end{split}
\]
Note that it is the restriction of the powers to $D_0$ which allows
the interior upper estimate to control their $L^2$ norms. Taking the
minimum of the resulting lower bounds over the finitely many possible
degrees proves \eqref{eq:partialC0}.
\end{proof}

\section{Type I tangent flows}\label{sec:soliton}

\subsection{The shrinking limit at a fixed base point}
In this section, we prove Theorem~\ref{thm:tangent}. The shrinking
soliton limit itself is provided by Bamler's tangent-flow and
structure theorems \cite[Theorems~2.18 and 2.37]{BamStructure}. The
main issue is to pass from $\mathbb F$ convergence to pointed
Gromov-Hausdorff convergence, and for this we follow the method of
Jian, Song and Tian \cite[Theorem~7.3]{JST} and of Hallgren, Jian,
Song and Tian \cite[Theorems~4.3-4.4]{HJST}. Under the global scalar
curvature bound of Part I, the required heat-kernel lower bound and
metric-measure convergence criterion are supplied by Hallgren
\cite[Lemma~3.2]{HallgrenEntropy}, \cite[Proposition~2.7]{HallgrenGW},
while the time continuity follows from the work of Bamler and Zhang
\cite{BZII}. We note that the arguments in this section are
independent of the partial $C^0$-estimate of
Section~\ref{sec:partialC0}.

Fix $x\in X$ and a conjugate heat family $\mu^x=(\mu^x_a)_{a<1}$
obtained as a limit of heat kernels with poles $(x,t_j)$ as
$t_j\rightarrow1$ (see \cite[Lemma~3.2]{HallgrenEntropy}). Then
Lemma~\ref{lem:moment} gives
\begin{equation}\label{eq:singulartimemoment}
 \int_X d_a(x,z)^2d\mu^x_a(z)\le C(1-a).
\end{equation}
In particular, its variance tends to zero as $a\rightarrow1$, as
required by \cite[Definition~2.32]{BamStructure}. As in the
introduction, set
$g_i(s)=\tau_i^{-1}g(1+\tau_i s)$ and
$\nu_{i,s}=\mu^x_{1+\tau_i s}$.

\begin{proposition}\label{prop:metricsoliton}
After passing to a subsequence, the metric-flow pairs
$(X,g_i(s),\nu_{i,s})$ converge in $\mathbb F$ to a shrinking metric
soliton with complete singular time slices, and the convergence is
smooth on compact subsets of the regular spacetime, where
\begin{equation}\label{eq:solitonequations}
 \Ric_{g_\infty(s)}+\nabla^2 f_s=\frac{1}{2|s|}g_\infty(s),
 \qquad (\nabla^2 f_s)^{2,0}=0.
\end{equation}
Moreover, if $(Z,d)$ denotes the time $-1$ slice of the limit, then
along the same subsequence, for every fixed $s<0$, there exists
$z(s)\in Z$ such that
\[
 (X,d_{g_i(s)},x)\longrightarrow (Z,\sqrt{|s|}\,d,z(s))
       \quad\text{in pointed Gromov-Hausdorff topology}.
\]
Finally, the soliton identifications preserve the complex structures
on the regular loci.
\end{proposition}
\begin{proof}
Perelman's entropy lower bound and Bamler's tangent-flow theorem
\cite[Theorem~2.37]{BamStructure} give the shrinking $\mathbb F$
limit. The structure of its singular time slices and smooth
convergence on the regular spacetime follow from
\cite[Theorems~2.5, 2.14 and 2.18]{BamStructure}.
Since the parallel complex structures converge on the regular locus,
and $\Ric$ and $g_\infty$ are $J_\infty$-invariant, the soliton
equation also gives $(\nabla^2f_s)^{2,0}=0$.

To establish pointed Gromov-Hausdorff convergence, write
$u_{i,s}=d\nu_{i,s}/dV_{g_i(s)}$. Since Part I gives a global Type I
scalar curvature bound, Hallgren's Gaussian estimate
\cite[Lemma~3.2]{HallgrenEntropy}, after rescaling, yields
\[
 u_{i,s}(y)\ge\frac{1}{C|s|^n}
 \exp\!\left(-\frac{C d_{g_i(s)}(x,y)^2}{|s|}\right).
\]
On each compact negative time interval, the volume noncollapsing
recalled in Section~2.2 therefore implies, for every $D<\infty$ and
$0<r\le1$,
\[
 \nu_{i,s}\bigl(B_{g_i(s)}(y,r)\bigr)\ge c_{s,D}r^{2n}>0,
 \qquad y\in B_{g_i(s)}(x,D),
\]
uniformly in $i$. Moreover, \eqref{eq:singulartimemoment} gives
\begin{equation}\label{eq:rescaledmoment}
 \int_X d_{g_i(s)}(x,z)^2d\nu_{i,s}(z)\le C|s|.
\end{equation}
At each time at which the $\mathbb F$ convergence is timewise, the
metric-measure slices converge in the Gromov-Wasserstein $W_1$ sense,
and the limiting slice is a complete, locally compact length space
whose measure has full support. Its variance is finite by
\eqref{eq:rescaledmoment} and lower semicontinuity, so Hallgren's
criterion \cite[Proposition~2.7]{HallgrenGW} gives pointed
Gromov-Hausdorff convergence to the same slice, with base points at
uniformly bounded distance from $x$ by \eqref{eq:rescaledmoment}.
After passing to a further subsequence, we may therefore take the base
point to be $x$ itself. We then
extract a common subsequence on a countable dense set of such times.
On each compact negative time interval, the uniform temporal distance
estimates in \cite[Theorem~1.1 and Corollaries~1.3-1.4]{BZII} and the
time continuity of the soliton extend the pointed convergence along
this subsequence to every time in the interval.

Finally, \cite[Theorem~2.18(a)-(c)]{BamStructure} provides maps
$F_s:Z\to\mathcal X_s$, $F_{-1}=\mathrm{id}$, satisfying
\begin{equation}\label{eq:solitonidentification}
 d_s(F_s(p),F_s(q))=\sqrt{|s|}\,d(p,q),\qquad p,q\in Z.
\end{equation}
On the regular spacetime, these maps are generated by the vector field
$\partial_s-\nabla f_s$, which preserves the complex structure.
Letting $x_\infty(s)$ denote the limiting marked point on
$\mathcal X_s$ and setting $z(s)=F_s^{-1}(x_\infty(s))$, we obtain
\eqref{eq:intropointedGH} on the fixed space $Z$.
\end{proof}

We note that the marked points $z(s)$ may vary with $s$ in soliton
coordinates. If the original flow extends smoothly through time $1$
near $x$, its rescalings simply converge to the Gaussian shrinking
soliton.

\subsection{Normal analytic structure}
We now turn to the complex analytic structure of the limit. By
Proposition~\ref{prop:metricsoliton}, $Z=\mathcal X_{-1}$ is a
singular K\"ahler-Ricci shrinker arising from compact K\"ahler-Ricci
flows with a uniform lower bound for the pointed Nash entropy. The
theorem of Hallgren and Zhang \cite[Theorem~B(i),(iii)]{HZ} therefore
endows $Z$ with a normal complex analytic structure of pure dimension
$n$ with log terminal singularities. Moreover, the analytic and metric
topologies agree, and the analytic smooth locus coincides with the
metric regular locus $\mathcal R_Z$. The application of their theorem
to finite-time tangent flows is included in \cite[Section~2.2, case~(I)]{HZ}.

Since the structure sheaf is characterized locally by continuous
functions which are holomorphic on $\mathcal R_Z$, and since the maps
$F_s$ in \eqref{eq:solitonidentification} are homeomorphisms which are
biholomorphic on the regular loci, these maps and their inverses
preserve the structure sheaves. They therefore identify all
negative-time slices biholomorphically with $Z$.

Finally, by \cite[Theorem~B(iv)]{HZ}, each point of $Z$ has a unique
metric tangent cone, which is Ricci-flat and K\"ahler and whose
underlying space is an affine variety with log terminal singularities. 
This is the tangent-cone structure used in
Section~\ref{sec:partialC0}. Together with
Proposition~\ref{prop:metricsoliton}, these results complete the proof
of Theorem~\ref{thm:tangent}.

\end{document}